\documentclass[a4paper]{article}

\usepackage[english]{babel}

\usepackage{cite}

\usepackage[utf8]{inputenc}
\usepackage[T1]{fontenc}
\usepackage{lmodern}

\usepackage{hyperref}
\usepackage{amssymb,amsmath,amsthm}

\usepackage{todonotes}
\usepackage[shortlabels]{enumitem}
\setenumerate{label={\normalfont(\alph*)},topsep=4pt,itemsep=0pt} 
\usepackage{mathtools,bbm}
\usepackage[normalem]{ulem}
\usepackage{cancel}
\usepackage{verbatim}
\usepackage{tikz}
\usetikzlibrary{matrix}

\usepackage{a4}
\usepackage{multirow}
\usepackage[footnotesize,bf,centerlast]{caption}
\usepackage[small,euler-digits,icomma,OT1,T1]{eulervm}

\usepackage{xcolor,colortbl}
\definecolor{Gray}{gray}{0.80}
\definecolor{LightGray}{gray}{0.90}

\newcommand{\cD}{\mathcal{D}}

\newcommand{\cF}{\mathcal{F}}

\newcommand{\cM}{\mathcal{M}}

\newcommand{\cP}{\mathcal{P}}

\newcommand{\cX}{\mathcal{X}}

\newcommand{\fS}{\mathfrak{S}}

\newcommand{\bE}{\mathbb{E}}

\newcommand{\bR}{\mathbb{R}}

\newcommand{\PR}{\mathbb{P}}
\newcommand{\bONE}{\mathbbm{1}}

\newcommand{\dd}{ \mathrm{d}}

\renewcommand{\epsilon}{\varepsilon}

\newcommand{\vn}[1]{\left| \! \left| #1\right| \! \right|}

\newcommand{\ip}[2]{\langle #1,#2\rangle}

\numberwithin{equation}{section}

\newtheorem{theorem}{Theorem}[section]
\newtheorem{lemma}[theorem]{Lemma}
\newtheorem{proposition}[theorem]{Proposition}

\theoremstyle{definition}
\newtheorem{definition}[theorem]{Definition}

\newtheorem{remark}[theorem]{Remark}

\title{Strict continuity of the transition semigroup for the solution of a well-posed martingale problem}

\author{Richard C. Kraaij\thanks{Delft Institute of Applied Mathematics, Delft University of Technology, Van Mourik Broekmanweg 6, 2628 XE Delft, The Netherlands. \emph{E-mail address}: r.c.kraaij@tudelft.nl}}
\date{\today}

\begin{document}

\maketitle

\begin{abstract}
	In this note we connect the notion of solutions of a martingale problem to the notion of a strongly continuous and locally equi-continuous semigroup on the space of bounded continuous functions equipped with the strict topology.
	This extends the classical connection of semigroups to Markov processes that was used successfully in the context of compact spaces to the context of Polish spaces. 
	
	In addition, we consider the context of locally compact spaces and show how the transition semigroup on the space of functions vanishing at infinity can be extended to the space of bounded continuous functions.
	

\end{abstract}

%   60J25  	Continuous-time Markov processes on general state spaces
%	60J35  	Transition functions, generators and resolvents

\section{Introduction}

A key technique in the study of (Feller) Markov processes $t \mapsto \eta(t)$ on compact spaces is the use of strongly continuous semigroups $\{S(t)\}_{t \geq 0}$ and their generators. In this setting, the transition semigroup of conditional expectations
\begin{equation*}
S(t)f(x) := \bE[f(\eta(t)) \, | \, \eta(0) =x], \qquad \qquad x \in \cX,
\end{equation*}
maps $C(\cX)$ into $C(\cX)$ and is strongly continuous. The infinitesimal generator `$A = \partial_t S(t)|_{t = 0}$' encodes the behaviour of the Markov process. These techniques have been used since the 50's e.g. to study the convergence of processes via the convergence of generators \cite{Tr58,Sk58,Ku69,Ku74,EK86}.

For compact spaces, the cornerstone that allows the application of these functional analytic techniques to probability theory, and Markov process theory in particular, is the Riesz-representation theorem which states that the continuous dual space of $(C(\cX),\vn{\cdot})$ is the space of Radon measures $\cM(\cX)$ and that $C(\cX)$ equips the space of measures with the weak topology. 

Already in the context of a locally compact space this effective connection breaks down. Consider for example Brownian motion on $\bR$. The semigroup of transition operators is not strongly continuous on $C_b(\bR)$ as can be seen by considering the function $f(x) = \sin(x^2)$. 

In the locally compact setting, two analytic ways to resolve this issue are
\begin{enumerate}[(a)]
	\item weakening the topology on $C_b(\cX)$ and work with a locally convex topology instead of a Banach topology; \label{item:weaking_topology}
	\item restricting the space of functions to $C_0(\cX)$, as the continuous dual space of $(C(\cX),\vn{\cdot})$ also equals $\cM(\cX)$. \label{item:restrict_space}
\end{enumerate}
Even though $C_b(\cX)$ clearly has the benefit of having more test-functions to work with, the loss of the Banach property has lead the literature on Markov process to generally opt for the use of $C_0(\cX)$ instead. The vague topology induced by $C_0(\cX)$ on $\cM(\cX)$, however, is defective. It does not see whether mass vanishes at infinity. Thus using strongly continuous semigroups on $C_0(\cX)$ to study Markov processes on a locally compact space comes at the cost of ad-hoc control of tightness\footnote{Tightness and relative weak compactness of sets of probability measures is equivalent by the Prohorov theorem \cite[Theorem 8.6.2]{Bo07}}.

It was soon realized by Stroock and Varadhan \cite{SV69a,SV69b,SV79} followed by \cite{PaStVa77,Kus80} that a, now widely used, probabilistic technique called
\begin{enumerate}[resume]
	\item the martingale problem \label{item:martingale problem}
\end{enumerate}
is a third way to study Markov processes via its infinitesimal characterization. We briefly give the definition of the martingale problem. Let $\cD(A)$ be a linear operator on $C_b(\cX)$. Then a process $\eta$ solves the martingale problem if
\begin{equation*}
t \mapsto f(\eta(t)) - f(\eta(0)) - \int_0^t Af(\eta(s)) \dd s
\end{equation*}
is a martingale with respect to the filtration $\cF_t := \sigma(\eta(s) \, | \, s \leq t)$ for all $f \in \cD(A)$. It can be shown that the solution of a well-posed martingale problem corresponds to a Markov process.

\smallskip

As this method uses probabilistic methods, it does not suffer from the functional analytic problems mentioned above and works in the general context of Polish spaces. It can, however, not directly tap in to the analytic structure that in the compact case was provided by the Riesz-representation theorem. 

\smallskip

The appropriate topology that corresponds to the weakening described in \ref{item:weaking_topology} in our trichotomy of options is called the \text{strict} topology. This topology was first introduced by Buck \cite{Bu58} in the 50's for locally compact $\cX$ and was studied in the 70's by Sentilles \cite{Se72} in the context of more general spaces, including Polish spaces. 

The essential property of the strict topology is that the Riesz representation continuous to hold: the continuous dual space of $(C_b(\cX),\beta)$ is the space of Radon measures. This way, the strict topology is the proper extension of the uniform topology from the compact to the Polish setting and has much greater potential to be applicable to the study of probability measures than the uniform topology. The topology also satisfies various other desirable properties that are known from the compact setting, see Appendix \ref{appendix:properties_strict_topology}.

\smallskip

That the strict topology can be used for a systematic study of Markov processes and their semigroups was shown by \cite{GK01,Ca11}. Semigroup theory in the context of topologies like the strict topology has been developed in the literature on so called bi-continuous semigroups, see e.g. \cite{Ku03,Fa11}, for more references on this area see also \cite{Kr16} in which a more topological, but to some extent equivalent, point of view is taken. An application of the strict topology in the study of properties of the propagation of stochastic order can be found in \cite{KrSc18}. 

\smallskip

This note aims at the resolution of the apparent trichotomy between \ref{item:weaking_topology}, \ref{item:restrict_space} and \ref{item:martingale problem} in the context of the strict topology. The two main results of this note are:
\begin{itemize}
	\item Theorem \ref{theorem:martingale_problem_gives SCLE_semigroup} shows that, in the context of Polish spaces, the transition semigroup of a collection of solutions to a well-posed martingale problem is strongly continuous and locally equi-continuous for the strict topology. The generator of this semigroup extends the operator of the martingale problem. This allows us to unify approach \ref{item:weaking_topology} and \ref{item:martingale problem}.
	\item Theorem \ref{theorem:transition_semigroup_on_locally_compact_is_strongly_continuous} shows that, in the context of locally compact Polish spaces, the transition semigroup on $C_0(\cX)$ can be extended, given that mass is conserved, to $C_b(\cX)$. This gives a partial merge of \ref{item:weaking_topology} and \ref{item:restrict_space}, modulo the known issue of having to avoid mass running of to infinity.
\end{itemize}
The proofs of these results are relatively straightforward. The straightforwardness, however, is a consequence of identifying the strict topology as being the appropriate setting, which can therefore be considered the main message of this note. An extension of these results from Polish spaces to a more general class of topological spaces, like the weak topology on a separable Hilbert space, should be accompanied by a better understanding of the appropriate locally convex structure for $C_b(\cX)$.

\section{Preliminaries: Strongly continuous semigroups for the strict topology and the martingale problem}\label{section:martingale_problem}

Let $\cX$ be a Polish space. We denote by $C_b(\cX)$ the space of continuous and bounded functions on $\cX$. If $\cX$ is locally compact, $C_0(\cX)$ is the space of continuous functions that vanish at infinity. $\cM(\cX)$ is the space of Radon measures on $\cX$. $\cP(\cX)$ is the subset of probability measures. 

\smallskip

Denote $\bR^+ = [0,\infty)$. Finally, the Skorokhod space $D_\cX(\bR^+)$ is the space of paths $\eta : \bR^+ \rightarrow \cX$ that have left limits and are right continuous. We equip $D_{\cX}(\bR^+)$ with the Skorokhod topology, see \cite{EK86}.

\smallskip

We say that a set $D \subseteq C_b(\cX)$ \textit{separates points} if for all $x,y \in \cX$ with $x \neq y$, there exists $f \in D$ such that $f(x) \neq f(y)$. We say that $D \subseteq C_b(\cX)$ \textit{vanishes nowhere} if for all $x \in \cX$ there exists $f \in D$ such that $f(x) \neq 0$.

\subsection{The martingale problem}

\begin{definition}[The martingale problem]
	Let $A \colon \cD(A) \subseteq C_b(\cX) \rightarrow C_b(\cX)$ be a linear operator. For $(A,\cD(A))$ and a measure $\nu \in \cP(\cX)$, we say that $\PR \in \cP(D_\cX(\bR^+))$ solves the \textit{martingale problem} for $(A,\nu)$ if $\PR \, \circ \, \eta(0)^{-1} = \nu$ and if for all $f \in \cD(A)$ 
	\begin{equation*}
	M_f(t) := f(\eta(t)) - f(\eta(0)) - \int_0^t Af(\eta(s)) \dd s
	\end{equation*}
	is a mean $0$ martingale with respect to its natural filtration $\cF_t := \sigma\left(\eta(s) \, | \, s \leq t \right)$ under $\PR$. 
	
	\smallskip
	
	We denote the set of all solutions to the martingale problem, for varying initial measures $\nu$, by $\cM_A$. We say that \textit{uniqueness} holds for the martingale problem if for every $\nu \in \cP(\cX)$ the set $\cM_\nu:= \{\PR \in \cM_A \, | \, \PR \, \circ \, \eta(0)^{-1} = \nu\}$ is empty or a singleton. Furthermore, we say that the martingale problem is \textit{well-posed} if this set contains exactly one element for every $\nu \in \cP(\cX)$.
\end{definition}

For later convenience, we will sometimes also write $A : \cD(A) \subseteq C_b(\cX) \rightarrow C_b(\cX)$ in terms of its graph: $A \subseteq C_b(\cX) \times C_b(\cX)$ where $(f,g) \in A$ if and only if $Af = g$.

%\smallskip

%The construction of solutions to the martingale problem can often be done via approximating processes. Classically, uniqueness for the martingale problem is proven via duality. \cite{CoKu15}, however, introduced a method based on viscosity solutions. 

\subsection{The strict topology} \label{section:strict_topology}

The strict topology on $C_b(\cX)$ is defined in terms of the \textit{compact-open} topology $\kappa$ on $C_b(\cX)$. This locally convex topology is generated by the semi-norms $p_K(f) := \sup_{x \in K} |f(x)|$, where $K$ ranges over all compact sets in $\cX$.

The \textit{strict} topology $\beta$ on the space bounded continuous functions $C_b(\cX)$  is generated by the semi-norms 
\begin{equation*}
p_{K_n,a_n}(f) := \sup_n a_n \sup_{x \in K_n} |f(x)|
\end{equation*}
varying over non-negative sequences $a_n$ converging to $0$ and sequences of compact sets $K_n \subseteq \cX$.

\begin{remark}
	We refer the reader to the discussion of the strict and sub-strict topology in \cite{Se72}, where it is shown that these two topologies coincide for Polish spaces. Because the definition of the sub-strict topology is more accessible, we use this as a characterisation of the strict topology in our context.
\end{remark}

\begin{remark}
	The strict topology can equivalently be given by the collection of semi-norms
	\begin{equation*} 
	p_g(f) := \vn{fg}
	\end{equation*}
	where $g$ ranges over the set 
	\begin{equation*}
	\left\{g \in C_b(\cX) \, | \, \forall \alpha > 0: \, \left\{x , | \,  |g(x)| \geq \alpha\right\} \text{ is compact in } \cX \right\}.
	\end{equation*}
	See \cite{Bu58} and \cite{Ca11}.
\end{remark}

\subsection{Semigroups} \label{section:semigroups_on_lcs}

We call the family of linear operators $\{T(t)\}_{t \geq 0}$ on $C_b(\cX)$ a \textit{semigroup} if $T(0) = \bONE$ and $T(t)T(s) = T(t+s)$ for $s,t \geq 0$. A family of $\beta$ to $\beta$ continuous operators $\{T(t)\}_{t \geq 0}$ is called a \textit{strongly continuous semigroup} if $t \mapsto T(t)f$ is $\beta$ continuous.

We call $\{T(t)\}_{t \geq 0}$ a \textit{locally equi-continuous} family if for every $t \geq 0$ and $\beta$-continuous semi-norm $p$ on $C_b(\cX)$ there exists a continuous semi-norm $q$ such that ${\sup_{s \leq t} p(T(s)f) \leq q(x)}$ for every $f \in E$. Finally, if $\{T(t)\}_{t \geq 0}$ is a locally equi-continuous strongly continuous semigroup we we say that $\{T(t)\}_{t \geq 0}$ is a \textit{SCLE} semigroup.

\smallskip

We say that the linear map $A \subseteq C_b(\cX) \times C_b(\cX)$ is the \textit{generator} of a SCLE semigroup $\{T(t)\}_{t \geq 0}$ if 
\begin{equation*}
g = \lim_{t \downarrow 0} \frac{T(t)f -f}{t} \quad \Leftrightarrow \quad (f,g) \in \cD(A).
\end{equation*}

If $A,B \subseteq C_b(\cX) \times C_b(\cX)$ are two operators such that $A \subseteq B$, we say that $B$ is an \textit{extension} of $A$.

For more information on \textit{SCLE} semigroups, see \cite{Kr16}.

\begin{remark}
	In the context of Banach spaces a strongly continuous semigroup is automatically locally equi-continuous.
\end{remark}

\section[The transition semigroup is SCLE]{The transition semigroup is strongly continuous and locally equi-continuous with respect to the strict topology} \label{section:martingale_problem_SCLE}

Under density and tightness conditions,Theorems 4.4.2 and 4.5.11 of \cite{EK86} state that if the martingale problem for $A$ has a unique solution on the Skorokhod space of trajectories on a Polish space $\cX$, then the semigroup of conditional expectations maps $C_b(\cX)$ into itself. The first main result of this paper is that this semigroup is in fact strongly continuous and locally equicontinuous for the strict topology $\beta$. 

A version of this result was also stated in time-inhomogeneous form in Theorems 2.12 and 2.13 of \cite{Ca11}, but the result below does not need \textit{sequential $\lambda$ dominance} or the \textit{Korovkin} condition. Note that the notion of a time-homogeneous martingale problem can be turned into a notion for time-homogeneous processes by adjoining a time variable to the state-space (see e.g. \cite[Section 4.7.A]{EK86}).

\begin{theorem}[Theorems 4.4.2 and 4.5.11 of \cite{EK86}] \label{theorem:martingale_problem_gives_semigroup}
	Let $A : \cD(A) \subseteq C_b(\cX) \times C_b(\cX)$ and let the martingale problem for $A$ be well-posed. 
	\begin{itemize}
		\item Suppose that the uniform closure of $\cD(A)$ contains an algebra $D$ that separates points and vanishes nowhere.
		\item Suppose that for all compact $K \subseteq \cX$, $\varepsilon >0$ and $T > 0$, there exists a compact set $K' = K'(K,\varepsilon,T) \subseteq \cX$ such that for all $\PR \in \cM_A$, we have
		\begin{equation} \label{eqn:theorem_compact_containment_set}
		\PR\left[\eta(t) \in K' \text{ for all } t < T, \eta(0) \in K \right] \geq (1-\varepsilon) \PR\left[\eta(0) \in K \right]. 
		\end{equation} 
	\end{itemize}
	Then the measures $\PR \in \cM_A$ correspond to strong Markov processes such that the operators $\{S(t)\}_{t \geq 0}$ on $C_b(\cX)$ defined by $S(t)f(x) = \bE[f(\eta(t))\, | \, \eta(0) = x]$ map $C_b(E)$ into $C_b(\cX)$. 
\end{theorem} 

\begin{theorem}\label{theorem:martingale_problem_gives SCLE_semigroup}
	Consider the context of Theorem \ref{theorem:martingale_problem_gives_semigroup}. The semigroup $\{S(t)\}_{t \geq 0}$ of conditional expectations is a $\beta$-SCLE semigroup. The generator of $\{S(t)\}_{t \geq 0}$ is an extension of $A$.
\end{theorem} 

The result follows from Lemmas \ref{lemma:semigroup_is_for_every_time_continuous} and \ref{lemma:semigroup_strongly_continuous} and Proposition \ref{proposition:generator_extends_martingale_problem_operator} below.

\begin{lemma} \label{lemma:semigroup_is_for_every_time_continuous}
	Let $\{S(t)\}_{t \geq 0}$ be the semigroup introduced in Theorem \ref{theorem:martingale_problem_gives SCLE_semigroup}. The family $\{S(t)\}_{t \geq 0}$ is locally equi-continuous for $\beta$.
\end{lemma}

\begin{proof}
	Fix $T \geq 0$. We will prove that $\{S(t)\}_{t \leq T}$ is $\beta$ equi-continuous by using Theorem \ref{theorem:propertiesCbstrict} (c) and (d). By (c) it suffices to establish sequential $\beta$ equi-continuity. Thus, pick a sequence $f_n$ converging to $f$ with respect to $\beta$. We aim to prove that $S(t)f_n$ converges to $S(t)f$ for $\beta$ uniformly for $t \leq T$.
	
	By Theorem \ref{theorem:propertiesCbstrict} (d), the strict topology equals the compact open topology on norm bounded sets. As $f_n$ converges to $f$ for the strict topology, we find $\sup_n \vn{f_n} \leq \infty$, which directly implies that $\sup_n \sup_{t \leq T} \vn{S(t)f_n} < \infty$. It therefore suffices to establish that $S(t)f_n$ converges to $S(t)f$ for the compact open topology, uniformly for $t \leq T$.
	
	\smallskip
	
	As $\beta$ is stronger than the compact-open topology, we have $f_n \rightarrow f$ uniformly on compact sets. Fix $\varepsilon > 0$ and a compact set $K \subseteq \cX$, and let $\hat{K}$ be the set introduced in Equation \eqref{eqn:theorem_compact_containment_set} for $T$. Then we obtain that
	\begin{align*}
	& \sup_{t \leq T} \sup_{x \in K} \left|S(t)f(x) - S(t)f_n(x) \right| \\
	& \leq \sup_{t \leq T}\sup_{x \in K}  \bE_x\left|f(\eta(t)) - f_n(\eta(t)) \right| \\
	& \leq  \sup_{t \leq T} \sup_{x \in K}  \bE_x\left|\left(f(\eta(t)) - f_n(\eta(t)) \right) \bONE_{\{\eta(t) \in \hat{K}\}} \right. \\
	& \qquad \qquad \qquad \qquad \qquad \left. + \left(f(\eta(t)) - f_n(\eta(t)) \right) \bONE_{\{\eta(t) \in \hat{K}^c\}} \right| \\
	& \leq \sup_{t \leq T} \sup_{y \in \hat{K}} \left|f(y) - f_n(y)\right| +  \sup_n \vn{f_n - f} \varepsilon.
	\end{align*}
	As $n \rightarrow \infty$ this quantity is bounded by $\sup_n \vn{f_n - f} \varepsilon$ as $f_n$ converges to $f$ uniformly on compacts. As $\varepsilon$ was arbitrary, we are done.
\end{proof}

\begin{lemma} \label{lemma:semigroup_strongly_continuous}
	Let $\{S(t)\}_{t \geq 0}$ be the semigroup introduced in Theorem \ref{theorem:martingale_problem_gives SCLE_semigroup}. Then $\{S(t)\}_{t \geq 0}$ is $\beta$ strongly continuous.
\end{lemma}

For the proof, we  recall  the notion of a weakly continuous semigroup. A semigroup is \textit{weakly continuous} if for all $f \in C_b(\cX)$ and $\mu \in \cM(\cX)$ the trajectory $t \mapsto \ip{S(t)f}{\mu}$ is continuous in $\bR$.

\begin{proof}[Proof of Lemma \ref{lemma:semigroup_strongly_continuous}]
	We first establish that it suffices to prove weak continuity of the semigroup by an application of Proposition 3.5 in \cite{Kr16}. To apply this proposition, note that we have  completeness and the strong Mackey property of $(C_b(\cX),\beta)$ by Theorem \ref{theorem:propertiesCbstrict} and local equi-continuity of the semigroup $\{S(t)\}_{t \geq 0}$ by Lemma \ref{lemma:semigroup_is_for_every_time_continuous}.
	
	Thus, we establish weak continuity. Let $f \in C_b(\cX)$ and $\mu \in \cM(\cX)$. Write $\mu$ as the Hahn-Jordan decomposition: $\mu = c^+ \mu^+ - c^-\mu^-$, where $c^+,c^- \geq 0$ such that $\mu^+, \mu^- \in \cP(\cX)$. It thus suffices to show that $t \mapsto \ip{S(t)f}{\mu^+}$ and $t \mapsto \ip{S(t)f}{\mu^-}$ are continuous. Clearly, it suffices to do this for either of the two.
	
	\smallskip
	
	Let $\PR$ be the unique solution to the martingale problem for $A$ started in $\mu^+$. It follows by Theorem 4.3.12 in \cite{EK86} that $\PR[X\eta(t) = \eta(t-)] = 1$ for all $t > 0$. Fix some $t \geq 0$, we show that our trajectory is continuous at this specific $t$. Note that
	\begin{equation*}
	\left|\ip{S(t)f}{\mu^+} - \ip{S(t+h)f}{\mu^+} \right|  \leq \bE^\PR\left|f(\eta(t) - f(\eta(t+h)) \right|. 
	\end{equation*}
	By the almost sure convergence of $\eta(t+h) \rightarrow \eta(t)$ as $h \rightarrow 0$, and the boundedness of $f$, we obtain by the dominated convergence theorem that this difference converges to $0$ as $h \rightarrow 0$. As $t \geq 0$ was arbitrary, the trajectory is continuous at all $t \geq 0$. 
\end{proof}

\begin{proposition} \label{proposition:generator_extends_martingale_problem_operator}
	Let $\{S(t)\}_{t \geq 0}$ be the semigroup introduced in Theorem \ref{theorem:martingale_problem_gives SCLE_semigroup} and let $\hat{A}$ be the generator of this semigroup. Then $\hat{A}$ is an extension of $A$.
\end{proposition}

\begin{proof}
	Let $f \in \cD(A)$, we prove that $f \in \cD(\hat{A})$. We again use the characterisation of $\beta$ convergence as given in Theorem \ref{theorem:propertiesCbstrict} (d). From this point onward, we write $g := Af$ to ease the notation.
	
	First, $\sup_{t} \vn{\frac{S(t) f - f}{t}} \leq \vn{g}$ as
	\begin{equation*}
	\frac{S(t)f(x) - f(x)}{t} = \bE_x\left[\frac{f(\eta(t)) - f(x)}{t}\right] = \bE_x \left[\frac{1}{t} \int_0^t g(\eta(s)) \dd s \right].
	\end{equation*}
	Second, we show that we have uniform convergence of $\frac{S(t) f - f}{t}$ to $g$ as $t \downarrow 0$ on compacts sets. So pick $K \subseteq \cX$ compact. Now choose $\varepsilon > 0$ arbitrary, and let $\hat{K} = \hat{K}(K,\varepsilon,1)$ as in \eqref{eqn:theorem_compact_containment_set}.
	\begin{align}
	& \sup_{x \in K} \left|\frac{S(t)f(x) - f(x)}{t} - g(x)\right|\notag \\ 
	& \leq \sup_{x \in K}  \bE_x \left[ \left|\frac{1}{t} \int_0^t g(\eta(s)) - g(x) \dd s \right| \right] \notag \\
	& \leq \sup_{x \in K}  \bE_x\left[ \bONE_{\{\eta(s) \in \hat{K} \text{ for } s \leq 1\}} \left|\frac{1}{t} \int_0^t g(\eta(s)) - g(x) \dd s \right| \right] \notag \\
	& \quad + \sup_{x \in K}  \bE_x \left[ \bONE_{\{\eta(s) \notin \hat{K} \text{ for } s \leq 1\}} \left|\frac{1}{t} \int_0^t g(\eta(s)) - g(x) \dd s \right| \right] \notag \\
	& \leq \sup_{x \in K}  \bE_x \left[ \bONE_{\{\eta(s) \in \hat{K} \text{ for } s \leq 1\}} \left|\frac{1}{t} \int_0^t g(\eta(s)) - g(x) \dd s \right| \right]  + 2\varepsilon \vn{g}. \label{eqn:bound_on_convergence_generator_martingale_problem}
	\end{align}
	Thus, we need to work on the first term on the last line. 
	
\smallskip
	
	The function $g$ restricted to the compact set $\hat{K}$ is uniformly continuous. So let $\varepsilon' > 0$, chosen smaller then $\varepsilon$, be such that if $d(x,y) < \varepsilon'$, $x,y \in \hat{K}$, then $|g(x) - g(y)| \leq \varepsilon$.
	
\smallskip
	
	By Lemma 4.5.17 in \cite{EK86}, the set $\{\PR_x \, | \, x \in K\}$ is a weakly compact set in $\cP(D_\cX(\bR^+))$. So by Theorem 3.7.2 in \cite{EK86}, we obtain that there exists a $\delta = \delta(\varepsilon') > 0$ such that
	\begin{equation*}
	\sup_{x \in K} \PR_x \left[\eta \in D_\cX(\bR^+) \, | \, \sup_{s \leq \delta} d(\eta(0),\eta(s)) < \varepsilon'\right] > 1 - \varepsilon' > 1 - \varepsilon.
	\end{equation*}
	Denote $S_\delta := \{\eta \in D_\cX(\bR^+) \, | \, \sup_{s \leq \delta} d(\eta(0),\eta(s)) < \varepsilon' \} $, so that we can summarize the equation as $\sup_{x \in K} \PR_x[S_\delta] > 1 - \varepsilon$.
	
\smallskip
	
	We reconsider the term that remained in equation \eqref{eqn:bound_on_convergence_generator_martingale_problem}. 
	\begin{align*}
	& \sup_{x \in K}  \bE_x \left[\bONE_{\{\eta(s) \in \hat{K} \text{ for } s \leq 1\}} \left|\frac{1}{t} \int_0^t g(\eta(s)) - g(x) \dd s \right| \right]  + 2 \varepsilon  \vn{g} \\
	& \leq \sup_{x \in K}  \bE_x \left[\bONE_{\{\eta(s) \in \hat{K} \text{ for } s \leq 1\} \cap S_\delta} \left|\frac{1}{t} \int_0^t g(\eta(s)) - g(x) \dd s \right| \right] + 4 \varepsilon \vn{g}. 
	\end{align*}
	On the set $\{\eta(s) \in \hat{K} \text{ for } s \leq 1\} \cap S_\delta$, we know that $d(\eta(s),x) \leq \eta$ as long as $s \leq \delta$. Thus by the uniform continuity of $g$ on $\hat{K}$, we obtain $|g(\eta(s)) - g(x)| \leq \varepsilon$ if $s \leq \delta$. Hence:
	\begin{equation*}
	\sup_{t \leq 1 \wedge \delta(\varepsilon')} \sup_{x \in K} \left|\frac{S(t)f(x) - f(x)}{t} - g(x)\right| \leq \varepsilon + 4 \varepsilon \vn{g}.
	\end{equation*}
	As $\varepsilon > 0$ was arbitrary, it follows that $f \in \cD(\hat{A})$ and $Af = g = \hat{A}f$.
	
\end{proof}

\section{Transition semigroup for a process on a locally compact space} \label{section:semigroup_on_locally_compact_space}

We now turn to the context of locally compact Polish spaces $\cX$. In this context, the transition semigroup is often considered to act on $C_0(\cX)$ equipped with the supremum norm. On this space it acts as a strongly continuous semigroup. We have seen, however, that on any Polish space $\cX$ the space $(C_b(\cX),\beta)$ natural for semigroup theory. In the following theorem, we will show that we can go forward and backward between the two perspectives.

The key property that allows to do this transition is that the Markov process under consideration conserves mass and its dual point of view, namely that the corresponding semigroup maps $C_0(\cX)$ into $C_0(\cX)$.

\begin{theorem} \label{theorem:transition_semigroup_on_locally_compact_is_strongly_continuous}
	Let $\{S(t)\}_{t \geq 0}$ be a SCLE semigroup on $(C_b(\cX),\beta)$ such that $S(t) C_0(\cX) \subseteq C_0(\cX)$ for every $t \geq 0$. Then the restriction of the semigroup to $C_0(\cX)$, denoted by $\{\tilde{S}(t)\}_{t \geq 0}$ is $\vn{\cdot}$ strongly continuous.
	
	\smallskip
	
	Conversely, suppose that we have a strongly continuous semigroup $\{\tilde{S}(t)\}_{t \geq 0}$ on $(C_0(\cX),\vn{\cdot})$ such that $\tilde{S}'(t) \cP(\cX) \subseteq \cP(\cX)$. Then the semigroup can be extended uniquely to a SCLE semigroup $\{S(t)\}_{t \geq 0}$ on $(C_b(\cX),\beta)$.
	
	\smallskip
	
	In this setting, denote by $(A,\cD(A))$ the generator of $\{S(t)\}_{t \geq 0}$ on $(C_b(\cX),\beta)$ and by $(\tilde{A},\cD(\tilde{A}))$ the generator of $\{\tilde{S}(t)\}_{t \geq 0}$ on $(C_0(\cX),\vn{\cdot})$. Then $\tilde{A} \subseteq A$ and $A$ is the $\beta$ closure of $\tilde{A}$.
\end{theorem}

Before we start with the proof, we note that both the space $(C_b(\cX),\beta)$ and $(C_0(\cX),\vn{\cdot})$ have the space of Radon measures as a dual. As such, the space of Radon measures caries two weak topologies. The first one is the one that probabilist call the \textit{weak} topology, i.e.~$\sigma(\cM(\cX),C_b(\cX))$, and the second is the weaker \textit{vague} topology, i.e.~$\sigma(\cM(\cX),C_0(\cX))$.

\begin{proof}
	\textit{Proof of the first statement.}

	For a given time $t \geq 0$, the operator $S(t)$ is continuous on $(C_0(\cX),\vn{\cdot})$, because $S(t)$ is $\beta$ continuous and therefore maps $\beta$-bounded sets into $\beta$-bounded sets. Norm continuity of the restriction $\tilde{S}(t)$ on $C_0(\cX)$ then follows by the fact that the bounded sets for the norm and for $\beta$ coincide.
	
	\smallskip
	
	As $\{S(t)\}_{t \geq 0}$ is $(C_b(\cX), \beta)$ is strongly continuous, it is also weakly continuous, in other words, for every Radon measure $\mu$, we have that 
	\begin{equation*}
	t \mapsto \ip{S(t)f}{\mu}
	\end{equation*}
	is continuous for every $f \in C_b(\cX)$ and in particular for $f \in C_0(\cX)$. Theorem I.5.8 in Engel and Nagel \cite{EN00} yields that the semigroup $\{\tilde{S}(t)\}_{t \geq 0}$ is strongly continuous on $(C_0(\cX),\vn{\cdot})$.
	
	\smallskip

	\textit{Proof of the second statement.}
	
	First note that as $C_0(\cX)$ is $\beta$ dense in $C_b(\cX)$ by the Stone-Weierstrass theorem, cf.~Theorem \ref{theorem:propertiesCbstrict} (e). It follows that there is at most one $\beta$-continuous extension of the semigroup $\overline{S}(t)$. We thus proceed to show that $\tilde{S}(t)$ is $\beta$ continuous, so that this extension exists by continuity.  In fact, we will directly prove the stronger statement that $\{\tilde{S}(t)\}_{t \geq 0}$ is locally $\beta$ equi-continuous. 
	
	\smallskip
	
	First of all, by the completeness of $(C_b(\cX),\beta)$, the fact that $C_0(\cX)$ is dense in $(C_b(\cX),\beta)$ and 21.4.(5) in \cite{Ko69}, we have $(C_0(\cX),\beta)' = (C_b(\cX),\beta)' = \cM(\cX)$ and the equi-continuous sets in $\cM(\cX)$ with respect to $(C_0(\cX),\beta)$ and $(C_b(\cX),\beta)$ coincide. It follows by 39.3.(4) in \cite{Ko79} that $\{\tilde{S}(t)\}_{t \geq 0}$ is locally $\beta$ equi-continuous if for every $T \geq 0$ and $\beta$ closed equi-continuous set $K \subseteq \cM(\cX)$ we have that 
	\begin{equation*}
	\fS K := \{\tilde{S}'(t)\mu \, | \, t \leq T, \mu \in K \}
	\end{equation*}
	is $\beta$ equi-continuous. By Theorem 6.1 (c) in Sentilles \cite{Se72}, it is sufficient to prove this result for $\beta$ equi-continuous sets $K$ consisting of non-negative measures in $\cM(\cX)$. 	Finally, it suffices to establish that $\fS K$ is weakly compact as this implies that $\fS K$ is $\beta$ equicontinuous by the strong Mackey property of $\beta$, see Theorem \ref{theorem:propertiesCbstrict} (a).

	\smallskip
	
	Thus, let $K$ be an arbitrary weakly closed $\beta$ equi-continuous subset of the non-negative Radon measures. By Theorem 8.9.4 in \cite{Bo07}, the weak topology on the positive cone in $\cM(E)$ is metrizable so it suffices to establish weak sequential compactness of $\fS K$. Let $\nu_n$ be a sequence in $\fS K$. Clearly, $\nu_n = \tilde{S}'(t_n) \mu_n$ for some sequence $\mu_n \in K$ and $t_n \leq T$. As $K$ is $\beta$ equi-continuous, it is weakly compact by the Bourbaki-Alaoglu theorem, so without loss of generality we restrict to a weakly converging subsequence $\mu_n \in K$ with limit $\mu \in K$ and $t_n \rightarrow t$, for some $t \leq T$.
	
	\smallskip
	
	Now there are two possibilities, either $\mu = 0$, or $\mu \neq 0$. In the first case, we obtain directly that $\nu_n = \tilde{S}'(t_n) \mu_n \rightarrow 0 \ni K \subseteq \fS K$ weakly. In the second case, one can show that
	\begin{equation*}
	\hat{\mu}_n := \frac{\mu_n}{\ip{\bONE}{\mu_n}} \rightarrow \frac{\mu}{\ip{\bONE}{\mu}} =: \hat{\mu}
	\end{equation*}
	weakly, and therefore vaguely. As $\{\tilde{S}(t)\}_{t \geq 0}$ is strongly continuous on $(C_0(\cX),\vn{\cdot})$, it follows that $\tilde{S}'(t_n)\hat{\mu}_n \rightarrow \tilde{S}'(t) \hat{\mu}$ vaguely. By assumption, all measures involved are probability measures, so by Proposition 3.4.4 in Ethier and Kurtz \cite{EK86} implies that the convergence is also in the weak topology. By an elementary computation, we infer that the result also holds without the normalising constants: $\nu_n \rightarrow \tilde{S}'(t) \mu$ weakly. 
	
	So both cases give us a weakly converging subsequence in $\fS K$.
	
	\smallskip
	
	We conclude that $\{\tilde{S}(t)\}_{t \leq T}$ is $\beta$ equi-continuous. So we can extend all $\tilde{S}(t)$ by continuity to $\beta$ locally equicontinuous semigroup of maps $S(t) \colon C_b(\cX) \rightarrow C_b(\cX)$. The last thing to show is the $\beta$ strong continuity of $\{S(t)\}_{t \geq 0}$.
	
	\smallskip
	
	By Proposition 3.5 in \cite{Kr16} it is sufficient to show weak continuity of the semigroup $\{S(t)\}_{t \geq 0}$. Pick $\mu \in \cM(\cX)$, and represent $\mu$ as the Hahn-Jordan decomposition $\mu = \mu^+ - \mu^-$, where $\mu^+,\mu^-$ are non-negative measures. By construction, the adjoints of $S(t)$ and $\tilde{S}(t)$ coincide, so $t \mapsto S'(t)\mu^+$ and $t \mapsto S'(t)\mu^-$ are vaguely continuous. The total mass of the measures in both trajectories remains constant by the assumption of the theorem, so by Proposition 3.4.4 in \cite{EK86}, we obtain that $t \mapsto S'(t)\mu^+$ and $t \mapsto S'(t)\mu^-$ are weakly continuous. This directly implies that $\{S(t)\}_{t \geq 0}$ is weakly continuous and thus strongly continuous. 
	
	\smallskip

	\textit{Proof of the third statement.}

	Let $(\tilde{A},\cD(\tilde{A}))$ be the generator of $\{\tilde{S}(t)\}_{t \geq 0}$ and $(A,\cD(A))$ the one of $\{S(t)\}_{t \geq 0}$. As the norm topology is stronger than $\beta$, it is immediate that $\tilde{A} \subseteq A$.
	
	We will show that $\cD(\tilde{A})$ is a \textit{core} for $(A,\cD(A))$, i.e.~$\cD(\tilde{A})$ is dense in $\cD(A)$ for the $\beta$-graph topology on $\cD(A)$. This will follow by a variant of Proposition II.1.7 in \cite{EN00} which proven for Banach spaces but which also holds for the strict topology.
	
	Thus, it suffices to prove $\beta$ density of $\cD(\tilde{A})$ in $C_b(\cX)$ and that $S(t) \cD(\tilde{A}) \subseteq \cD(\tilde{A})$. 
	
	The first claim follows because $\cD(\tilde{A})$ is norm, hence $\beta$, dense in $C_0(\cX)$ by Theorem II.1.4 in \cite{EN00} and because $C_0(\cX)$ is $\beta$ dense in $C_b(\cX)$ by the Stone-Weierstrass theorem, cf. Theorem \ref{theorem:propertiesCbstrict} (e). The second claim follows because $S(t) \cD(\tilde{A}) = \tilde{S}(t) \cD(\tilde{A}) \subseteq \cD(\tilde{A})$ by e.g.~Lemma II.1.3 in \cite{EN00} or Lemma 5.2 in \cite{Kr16}. 
	
	\smallskip
	
	We conclude that $\cD(\tilde{A})$ is a core for $\cD(A)$. As $A$ is $\beta$ closed, it follows that $A$ is the $\beta$ graph-closure of $\tilde{A}$.
\end{proof}

\appendix

\section{Properties of the strict topology} \label{appendix:properties_strict_topology}

The strict topology is the `right' generalisation of the norm topology on $C(\cX)$ for compact metric $\cX$ to the more general context of Polish spaces. To avoid further scattering of results, we collect some of the main properties of $\beta$.

\begin{theorem} \label{theorem:propertiesCbstrict} 
	Let $\cX$ be Polish. The locally convex space $(C_b(\cX),\beta)$ satisfies the following properties.
	\begin{enumerate}[(a)]
		\item $(C_b(\cX),\beta)$ is complete, strong Mackey (i.e.~all weakly compact sets in the dual are equi-continuous) and the continuous dual space coincides with the space of Radon measures on $\cX$ of bounded total variation.
		\item $(C_b(\cX),\beta)$ is separable.
		\item For any locally convex space $(F,\tau_F)$ and $\beta$ to $\tau_F$ sequentially equi-continuous family $\{T_i\}_{i \in I}$ of maps $T_i \colon (C_b(\cX),\beta) \rightarrow (F,\tau_F)$,  the family $I$ is $\beta$ to $\tau_F$ equi-continuous.
		\item The norm bounded and $\beta$ bounded sets coincide. Furthermore, on norm bounded sets $\beta$ and $\kappa$ coincide.
		\item Stone-Weierstrass: Let $M$ be an algebra of functions in $C_b(\cX)$. If $M$ vanishes nowhere and separates points, then $M$ is $\beta$ dense in $C_b(\cX)$.
		\item Arzel\`{a}-Ascoli: A set $M \subseteq C_b(\cX)$ is $\beta$ compact if and only if $M$ is norm bounded and $M$ is an equi-continuous family of functions.
		\item $(C_b(\cX),\beta,\leq)$, where $\leq $ is defined as $f \leq g$ if and only if $f(x) \leq g(x)$ for all $x \in \cX$, is locally convex-solid.
		\item Dini's theorem: If $\{f_\alpha\}_{\alpha}$ is a net in $C_b(\cX)$ such that $f_\alpha$ increases or decreases point-wise to $f \in C_b(\cX)$, then $f_\alpha \rightarrow f$ for the strict topology.
	\end{enumerate}
\end{theorem}

Note that $(d)$ implies that a sequence $f_n \stackrel{\beta}{\rightarrow} f$ if and only if $\sup_n \vn{f_n} < \infty$ and $f_n \stackrel{\kappa}{\rightarrow} f$. 

\begin{proof}
	(a) and (c) follow from Theorems 9.1 and 8.1 in \cite{Se72}, Theorem 7.4 in \cite{Wi81}, Corollary 3.6 in \cite{We68} and Krein's theorem\cite[24.5.(4)]{Ko69}. (b) follows from Theorem 2.1 in \cite{Su72}. (d) follows by Theorems 4.7, 2.4 in \cite{Se72} and 2.2.1 in \cite{Wi61}. (e) is proven in Theorem 2.1 and Corollary 2.4 in \cite{Ha76}. (f) follows by the Arzel\`{a}-Ascoli theorem for the compact-open topology, Theorem 8.2.10 in \cite{En89}, and (d). To conclude, (g) and (h) follow from Theorems 6.1 and 6.2 in \cite{Se72}.
\end{proof}

\textbf{Acknowledgment}

The author thanks Moritz Schauer for discussions on the topic of the strict topology and its applicability to probability theory.

%
%
%The following result is an analogue of the Banach-Steinhaus theorem for the setting of strong Mackey spaces and is useful for the study of operators for the strict topology. The proof follows from Lemma 3.8 in \cite{Ku09} and the fact that the strict topology is strong Mackey. See also Lemma 3.2 in \cite{Kr16} for the use of a variant of this result in the setting of semigroups. %We produce the proof here for completeness and readability. 
%
%\begin{proposition} \label{proposition:equi_continuity_via_strong_continuity}
%	Suppose $\{T_n\}_{n \geq 1}$ is a family of strictly continuous linear operators from $C_b(\cX)$ to $C_b(\cX)$. Additionally assume that $T_n f \rightarrow f$ for all $f \in C_b(\cX)$. Then the family $\{T_n\}$ is equi-continuous: for every $\beta$ continuous semi-norm $p$ there is a $\beta$ continuous semi-norm $q$ such that
%	\begin{equation*}
%	\sup_{n \geq 1} p\left(T_n f \right) \leq q(f)
%	\end{equation*}
%	for all $f \in C_b(\cX)$.
%\end{proposition}

%\printbibliography
\bibliographystyle{abbrv}
\bibliography{../KraaijBib}

\begin{thebibliography}{10}

\bibitem{Bo07}
V.~I. Bogachev.
\newblock {\em Measure Theory}.
\newblock Springer-Verlag, 2007.

\bibitem{Bu58}
R.~C. Buck.
\newblock Bounded continuous functions on a locally compact space.
\newblock {\em Michigan Math. J.}, 5(2):95--104, 1958.

\bibitem{EN00}
K.-J. Engel and R.~Nagel.
\newblock {\em One-Parameter Semigroups for Linear Evolution Equations}.
\newblock Springer-Verlag, 2000.

\bibitem{En89}
R.~Engelking.
\newblock {\em General topology}.
\newblock Heldermann Verlag, Berlin, second edition, 1989.

\bibitem{EK86}
S.~N. Ethier and T.~G. Kurtz.
\newblock {\em Markov processes: Characterization and Convergence}.
\newblock Wiley, 1986.

\bibitem{Fa11}
B.~Farkas.
\newblock Adjoint bi-continuous semigroups and semigroups on the space of
  measures.
\newblock {\em Czechoslovak Mathematical Journal}, 61(2):309--322, 2011.

\bibitem{GK01}
B.~Goldys and M.~Kocan.
\newblock Diffusion semigroups in spaces of continuous functions with mixed
  topology.
\newblock {\em J. Differential Equations}, 173(1):17--39, 2001.

\bibitem{Ha76}
R.~G. Haydon.
\newblock On the {S}tone-{W}eierstrass theorem for the strict and superstrict
  topologies.
\newblock {\em Proc. Amer. Math. Soc.}, 59(2):273--278, 1976.

\bibitem{Ko69}
G.~K{\"o}the.
\newblock {\em Topological Vector Spaces I}.
\newblock Springer-Verlag, 1969.

\bibitem{Ko79}
G.~K{\"o}the.
\newblock {\em Topological Vector Spaces II}.
\newblock Springer-Verlag, 1979.

\bibitem{Kr16}
R.~C. Kraaij.
\newblock Strongly continuous and locally equi-continuous semigroups on locally
  convex spaces.
\newblock {\em Semigroup Forum}, 92(1):158--185, 2016.

\bibitem{KrSc18}
R.~C. Kraaij and M.~Schauer.
\newblock A generator approach to stochastic monotonicity and propagation of
  order.
\newblock {\em preprint; ArXiv:1804.10222}, 2018.

\bibitem{Ku03}
F.~K\"{u}hnemund.
\newblock A {H}ille-{Y}osida theorem for bi-continuous semigroups.
\newblock {\em Semigroup Forum}, 67(2):205--225, 2003.

\bibitem{Ku69}
T.~{Kurtz}.
\newblock {Extensions of Trotter's operator semigroup approximation theorems.}
\newblock {\em {J. Funct. Anal.}}, 3:354--375, 1969.

\bibitem{Ku74}
T.~G. {Kurtz}.
\newblock {Convergence of sequences of semigroups of nonlinear operators with
  an application to gas kinetics.}
\newblock {\em {Trans. Am. Math. Soc.}}, 186:259--272, 1974.

\bibitem{Kus80}
H.~J. Kushner.
\newblock A martingale method for the convergence of a sequence of processes to
  a jump-diffusion process.
\newblock {\em Z. Wahrsch. Verw. Gebiete}, 53(2):207--219, 1980.

\bibitem{PaStVa77}
G.~C. Papanicolaou, D.~Stroock, and S.~S. Varadhan.
\newblock Martingale approach to some limit theorems.
\newblock In {\em Duke Turbulence Conference (Duke Univ., Durham, NC, 1976),
  Paper}, volume~6, 1977.

\bibitem{Se72}
F.~D. Sentilles.
\newblock Bounded continuous functions on a completely regular space.
\newblock {\em Trans. Amer. Math. Soc.}, 168:311--336, 1972.

\bibitem{Sk58}
A.~V. Skorohod.
\newblock Limit theorems for {M}arkov processes.
\newblock {\em Teor. Veroyatnost. i Primenen.}, 3:217--264, 1958.

\bibitem{SV69a}
D.~W. Stroock and S.~R.~S. Varadhan.
\newblock Diffusion processes with continuous coefficients, i.
\newblock {\em Communications on Pure and Applied Mathematics}, 22(3):345--400,
  1969.

\bibitem{SV69b}
D.~W. Stroock and S.~R.~S. Varadhan.
\newblock Diffusion processes with continuous coefficients, ii.
\newblock {\em Communications on Pure and Applied Mathematics}, 22(4):479--530,
  1969.

\bibitem{SV79}
D.~W. Stroock and S.~R.~S. Varadhan.
\newblock {\em Multidimensional diffusion processes}, volume 233 of {\em
  Grundlehren der Mathematischen Wissenschaften [Fundamental Principles of
  Mathematical Sciences]}.
\newblock Springer-Verlag, Berlin-New York, 1979.

\bibitem{Su72}
W.~H. Summers.
\newblock Separability in the strict and substrict topologies.
\newblock {\em Proc. Amer. Math. Soc.}, 35:507--514, 1972.

\bibitem{Tr58}
H.~F. Trotter.
\newblock Approximation of semi-groups of operators.
\newblock {\em Pacific J. Math.}, 8:887--919, 1958.

\bibitem{Ca11}
J.~A. van Casteren.
\newblock {\em Markov processes, {F}eller semigroups and evolution equations},
  volume~12 of {\em Series on Concrete and Applicable Mathematics}.
\newblock World Scientific Publishing Co. Pte. Ltd., Hackensack, NJ, 2011.

\bibitem{We68}
J.~H. Webb.
\newblock Sequential convergence in locally convex spaces.
\newblock {\em Mathematical Proceedings of the Cambridge Philosophical
  Society}, 64:341--364, 4 1968.

\bibitem{Wi81}
A.~Wilansky.
\newblock Mazur spaces.
\newblock {\em {Int. J. Math. Math. Sci.}}, 4:39--53, 1981.

\bibitem{Wi61}
A.~Wiweger.
\newblock Linear spaces with mixed topology.
\newblock {\em Studia Math.}, 20:47--68, 1961.

\end{thebibliography}
\end{document}